\documentclass[a4paper]{article}     % onecolumn (second format)
\usepackage[latin5]{inputenc}
\usepackage{graphicx}
\usepackage{amsmath}
\usepackage{amsthm}
\usepackage{amssymb}
\usepackage{amsthm}
\usepackage{mathtools}
\usepackage{vmargin}
\usepackage{mathptmx}      % use Times fonts if available on your TeX system
\newtheorem{theorem}{Theorem}[section]

\newtheorem{corollary}[theorem]{Corollary}
\newtheorem{proposition}[theorem]{Proposition}
\newtheorem{definition}[theorem]{Definition}
\newtheorem{example}[theorem]{Example}

\setmarginsrb{4cm}{3cm}{2.5cm}{1cm}{0cm}{1cm}{1cm}{1cm}
\usepackage[affil-it]{authblk}
\usepackage{fancyhdr}
\begin{document}
\lhead{G. Budak\c{c}\i \; \& H. Oru\c{c}}
\rhead{$q$-Peano Kernel and Its Applications}
%\markboth{}{$q$-Peano Kernel and Its Applications}
%\markleft{G. Budak\c{c}\.ı\& H. Oru\c{c}}

\title{$q$-Peano Kernel and Its Applications
% General acknowledgments should be placed at the end of the article.
}


%\author{G\"{u}lter Budak\c{c}\i \\ \itshape
%                  \and
%        Halil Oru\c{c} %etc.
%}
%\author{G\"{u}lter Budak\c{c}\i}
%\email{gulter.budakci@deu.edu.tr}
%\affiliation{Department of Mathematics, Dokuz Eyl\"ul University, Fen Bilimleri Enstit\"us\"u,
% T{\i}naztepe Kamp\"us\"u, 35390 Buca, \.Izmir }
%\author{Halil Oru\c{c}}%
%\email{halil.oruc@deu.edu.tr}
%\thanks{Corresponding author}
%\affiliation{Department of Mathematics, Dokuz Eyl\"ul University, Fen Fak\"ultesi,
% T{\i}naztepe Kamp\"us\"u, 35390 Buca, \.Izmir}

\author{G\"{u}lter Budak\c{c}\i%
  \thanks{Electronic address: \texttt{gulter.budakci@deu.edu.tr}}}
\affil{Department of Mathematics, Dokuz Eyl\"ul University \\ Fen Bilimleri Enstit\"us\"u,
 T{\i}naztepe Kamp\"us\"u, 35390 Buca, \.Izmir }
 
 \author{Halil Oru\c{c}%
   \thanks{Electronic address: \texttt{halil.oruc@deu.edu.tr}; Corresponding author}}
 \affil{Department of Mathematics, Dokuz Eyl\"ul University \\ Fen Fak\"ultesi,
  T{\i}naztepe Kamp\"us\"u, 35390 Buca, \.Izmir}

%\institute{G\"{u}lter Budak\c{c}\i \at
%             Department of Mathematics, Dokuz Eyl\"ul University, Fen Bilimleri Enstit\"us\"u,
%                           T{\i}naztepe Kamp\"us\"u, 35390 Buca, \.Izmir   \\
%              %Tel.: +123-45-678910\\
%              %Fax: +123-45-678910\\
%              \email{gulter.budakci@deu.edu.tr}           %  \\
%%             \emph{Present address:} of F. Author  %  if needed
%           Halil Oru\c{c} \at
%              {Department of Mathematics, Dokuz Eyl\"ul University, Fen Fak\"ultesi,
%                            T{\i}naztepe Kamp\"us\"u, 35390 Buca, \.Izmir\\
%            \email{halil.oruc@deu.edu.tr} 
%}}


\date{}
\maketitle

\begin{abstract}
We introduce a $q$-analogue of the Peano kernel theorem by replacing ordinary derivatives and integrals by quantum derivatives and quantum integrals. In the limit $q\rightarrow 1$, the  $q$-Peano kernel reduces to the classical Peano kernel. We also give applications to polynomial interpolation and construct examples in which classical remainder theory fails whereas $q$-Peano kernel works. Furthermore we derive a relation between $q$-B-splines and divided differences via the $q$-Peano kernel.\\ 
%Include keywords and mathematical subject classification numbers as needed.
\end{abstract}

\section{Introduction}
The Peano kernel theorem provides a useful technique for computing the errors of approximations such as interpolation, quadrature rules and B-splines. The errors are represented by a linear functional that operates on functions $f \in {C}^{n+1}[a,b]$ and annihilates all polynomials of degree at most $n$.\\

\noindent Namely, if $L(f)=0$ for all $f\in \mathcal{P}_n$, the space of polynomials of degree $n$, then
\[ L(f)=\int\limits_{a}^{b}f^{(n+1)}(t)K(x,t)dt, \]
where $K(x,t)=\dfrac{1}{n!}L\left((x-t)_+^n\right).$

An important application of this result is the Kowalewski's interpolating polynomial remainder. Let $t_0,t_1,\ldots,t_n \in [a,b]$ be fixed and distinct, and  
\[ L(f)=f(x)-\sum\limits_{k=0}^{n}f(t_k)l_{nk}(x) \]
where $l_{nk}(x)=\prod\limits_{\substack{v=0 \\ v\neq k}}^{n}\dfrac{x-t_v}{t_k-t_v}$. If $f\in C^{m+1}[a,b]$, then
\[ L(f)=\frac{1}{m!}\sum\limits_{k=0}^{n}l_{nk}(x)\int\limits_{t_k}^{x}(t_k-t)^mf^{(m+1)}(t)dt, \quad \textrm{for each }\; m=0,1,\ldots,n \]
is the error functional, see \cite{hammer}. Our purpose is to extend the Peano kernel when classical derivatives are replaced by $q$-derivatives. This extension is important because there are  functions whose $q$-derivatives exist but whose classical derivatives  fail to exist.

Section 2 contains definitions and properties of the quantum calculus which we use in the next sections. In Section 3, we give the  $q$-Taylor theorem and develop a $q$-analogue of the Peano kernel ($q$-Peano kernel). Furthermore, we present a simple way to find the kernel under some conditions. Section 4 demonstrates how the $q$-Peano kernel is used to find the error of Lagrange interpolation. A $q$-analogue of the trapezoidal rule is also given. Moreover, we discuss the error bounds of quadrature formula on the remainder. Finally, we establish a relation between the $q$-B-splines and the $q$-Peano kernel in Section 5.

\section{Preliminaries}

We begin by giving basic definitions and theorems of the $q$-calculus that are required in the next section.

\noindent For a fixed parameter $q\neq 1$, the $q$-derivatives are defined by,
\begin{eqnarray} D_qf(t)&=&\frac{f(qt)-f(t)}{(q-1)t} \nonumber \label{qqderivative}\\
D_q^nf(t)&=&D_q(D_q^{n-1}f(t)), \quad n\geqslant 2.\nonumber
 \end{eqnarray}
Note that if $f$ is a differentiable function, then 
\begin{eqnarray*}
\lim\limits_{q\to 1} D_qf(x)&=&Df(x).\nonumber
\end{eqnarray*}
For polynomials the $q$-derivative is easy to compute. Indeed it follows easily
from the definition of the $q$-derivative that
\[ D_qx^n=[n]_qx^{n-1}, \]
where the $q$-integers $[n]_q$ are defined by,
 \begin{equation*}
 [n]_q=\left\{\begin{array}{lcc}
 (1-q^n)/(1-q), & q \neq 1,\\
 n, & q =1.
 \end{array}\right.
 \end{equation*}
Moreover, the $q$-factorial is defined by
\[ [n]_q!=[1]_q\cdots[n]_q. \]  

 Quantum integrals are the analogues of classical integrals for the quantum calculus. Quantum integrals satisfy a quantum version of the fundamental theorem of calculus, see \cite{quantumcalculusbook} for details.
 
 \begin{definition}
 Let $0<a<b$. Then the definite $q$-integral of a function $f(x)$ is defined by 
 \[ \int_{0}^{b}f(x)d_qx=(1-q)b\sum_{i=0}^{\infty}q^if(q^ib) \]
 and
 \[ \int_{a}^{b}f(x)d_qx=\int_{0}^{b}f(x)d_qx-\int_{0}^{a}f(x)d_qx. \]
 \end{definition}
 \begin{theorem}\label{funda}[Fundamental Theorem of Calculus]\\
 If $F(x)$ is continuous at $x=0$, then
 \[ \int_{a}^{b}D_qF(x)d_qx=F(b)-F(a)\]
 where $0\leqslant a<b\leqslant \infty.$
 \end{theorem}
 
 The work \cite{mean} gives the  mean value theorem in the $q$-calculus which will be needed in one of our results.
 
 \begin{theorem}\label{mean}
 If  $F$ is  continuous and $G$ is $1/q$-integrable and is nonnegative(or nonpositive) on $[a,b]$, then there exists $\tilde{q}\in (1,\infty)$ such that
for all  $q>\tilde{q}$ there exists a   $\xi\in (a,b)$  for which
\[ \int\limits_{a}^{b}F(x)G(x)d_{1/q}x=F(\xi)\int\limits_{a}^{b}G(x)d_{1/q}x.  \]
% \[ \left(\forall q\in (\tilde{q},\infty)\right)\left(\exists\xi\in (a,b)\right)\quad \int\limits_{a}^{b}F(x)G(x)d_{1/q}x=F(\xi)\int\limits_{a}^{b}G(x)d_{1/q}x. \] 
 \end{theorem}
 
 We also require a $q$-H\"older inequality and appropriate notions of distance in $q$-integrals, see \cite{intelligent}, \cite{gauchman} and \cite{tariboon}.
 
 \begin{definition}
  We will denote by $L_{p,q}([0,b])$ with $1\leqslant p<\infty$  the set of all functions $f$ on $[0,b]$ such that
  \[ ||f||_{p,q}:=\left(\int\limits_{0}^{b}|f|^pd_{1/q}t\right)^{\frac{1}{p}}<\infty.\]
Furthermore let $L_{\infty,q}([0,b])$ denote the set of all functions $f$ on $[0,b]$ such that
 \[ ||f||_{\infty,q}:=\sup\limits_{x\in[0,b]}|f(x)|<\infty.\]
 
% \begin{itemize}
% \item[i)] Let $L_{p,q}([0,b])$ with $1\leqslant p<\infty$ be the set of all functions on $[0,b]$ such that
%  \[ \left(\int\limits_{0}^{b}|f|^pd_{1/q}t\right)^{\frac{1}{p}}<\infty. \]
%  For a function $f\in L_{p,q }([0,b])$, 
%  \[ ||f||_{p,q}:=\left(\int\limits_{0}^{b}|f|^pd_{1/q}t\right)^{\frac{1}{p}}.\]
%  is called a norm on $L_{p,q}$.
%  \item[ii)]  Let $L_{\infty,q}([0,b])$ be the set of all functions defined on $[0,b]$ such that
%  \[ \sup\limits_{x\in[0,b]}|f(x)|<\infty. \]
%  For a function $f\in L_{\infty,q}([0,b])$,
%  \[ ||f||_{\infty,q}:=\sup\limits_{x\in[0,b]}|f(x)|. \]
% \end{itemize}
  \end{definition}

 % The next theorem states $q$-H\"older's inequality, for details see \cite{intelligent} and \cite{tariboon}. 
  \begin{theorem}
 Let $x\in [0,b]$, $q\in [1,\infty)$ and $p_1,p_2>1$ be such that $\frac{1}{p_1}+\frac{1}{p_2}=1$. Then 
 \[ \int\limits_{0}^{x}|f(x)||g(x)|d_{1/q}t\leqslant \left(\int\limits_{0}^{x}|f(x)|^{p_1}d_{1/q}t\right)^\frac{1}{p_1}\left(\int\limits_{0}^{x}|g(x)|^{p_2}d_{1/q}t\right)^\frac{1}{p_2}. \]
  
 \end{theorem}

\section{$q$-Peano Kernel Theorem}

In this section we derive a generalization of the Peano kernel theorem. This generalization is based on the $q$-Taylor expansion analogous to the proof of the  classical Peano kernel Theorem. So we start by giving the $q$-Taylor expansion with integral representation.
A detailed treatment of the classical Peano Kernel theorem can be found in \cite{hammer}, \cite{phillips} and \cite{powel}.

We use the notation $q$-$C^{k}[a,b]$ to denote the space of functions whose $q$-derivatives of order up to $k$ are continuous on $[a,b]$.\\

%A different form of the following theorem is derived in \cite{salam}.

\begin{theorem} ($q$-Taylor Theorem)
Let  $f$ be $n+1$ times $1/q$-differentiable in the closed interval $[a,b]$. Then
\begin{equation}\label{qtaylor}
f(x)=\sum_{k=0}^{n}q^{k(k-1)/2}\frac{(D_{1/q}^kf)(q^ka)}{[k]_q!}(x-a)^{k,q}+R_n(f), 
\end{equation} 
where
\[ (x-t)^{n,q}=(x-q^{n-1}t)\cdots(x-qt)(x-t) \]
and
\[ R_n(f)=\frac{q^{n(n+1)/2}}{[n]_q!}\int_a^x (D_{1/q}^{n+1}f)(q^nt)(x-t)^{n,q}d_{1/q}t.\]
%Notice that as in classical case we can also rewrite the remainder with truncated power functions:
Another way to express the remainder $R_nf$ is to employ  the truncated power function. That is
\begin{equation}\label{(3.2)}
 R_n(f)=\frac{q^{n(n+1)/2}}{[n]_q!}\int_a^b (D_{1/q}^{n+1}f)(q^nt)(x-t)_+^{n,q}d_{1/q}t,
\end{equation}
\vspace{0.5cm}
where
\[ (x-t)_+^{n,q}=(x-q^{n-1}t)\cdots(x-qt)(x-t)_+. \]
\end{theorem}
Here $(x-t)_+$ is the truncated power function
\[ (x-t)_+=\left\{\begin{array}{lr}
x-t, & \textrm{if}\;\;\;\;x>t\\
0, &\textrm{otherwise.}
\end{array}\right. \]
There are other forms of $q$-Taylor Theorem, see for example \cite{salam}, \cite{oktay}, \cite{mourad}. 
%The work \cite{oktay} introduces $q$-analytic and $q$-harmonic functions and solves the  quantum $q$-oscillator problem.
\begin{theorem}\label{peanothm}
Let 
$ g_t(x)=(x-t)_+^{n,q}$
and let $L$ be a linear functional that commutes with the operation of $q$-integration and also satisfies the conditions: $L(g_t)$ exists and $L(f)=0$ for all $f\in \mathcal{P}_n$. Then for all $f\in 1/q-C^{n+1}[a,b]$ 
\[ L(f)=\int_{a}^{b}(D_{1/q}^{n+1}f)(q^nt)K(x,t)d_{1/q}t, \]
where
\[ K(x,t)=\frac{q^{n(n+1)/2}}{[n]_q!}L(g_t). \]
\end{theorem}

\begin{proof}
Recall that here the function $(x-t)_+^{n,q}$ is a function of $t$ and  $x$ behaves as a parameter. When we say $L(g_t)$ we mean that $L$ is applied to the truncated power function, regarded as a function of $x$ with $t$ as a parameter. Hence we find real number that depends on $t$. We apply $L$ to the equation \eqref{qtaylor}. Since $L$ is linear and annihilates polynomials, we  have 
\[ L(f)=\frac{q^{n(n+1)/2}}{[n]_q!}L\left(\int_a^b (D_{1/q}^{n+1}f)(q^nt)(x-t)_+^{n,q}d_{1/q}t\right). \]
Since $L$ commutes with the operation of $q$-integration, 
\[ L(f)=\frac{q^{n(n+1)/2}}{[n]_q!}\int_a^b (D_{1/q}^{n+1}f)(q^nt)L\left((x-t)_+^{n,q}\right)d_{1/q}t.   \]
%notice that the linear functional $L$ is applied to $(x-t)_+^{n,q}$ as a function of $x$. 
\end{proof}

\begin{corollary}
If the conditions in Theorem \ref{peanothm} are satisfied and also the kernel  $K(x,t)$ does not change sign on $[a,b]$, then
\[ L(f)=\frac{\left(D_{1/q}^{n+1}f\right)(\xi)}{[n+1]_q!}q^{n(n+1)/2}L(x^{n+1}) \] 
\end{corollary}
\begin{proof}
Since $D_{1/q}^{n+1}f$ is continuous and $K(x,t)$ does not change sign on $[a,b]$, we can apply the Mean Value Theorem \ref{mean}. Thus we have
\[  L(f)=\left(D_{1/q}^{n+1}f\right)(\xi)\int\limits_{a}^{b}K(x,t)d_{1/q}t,\quad a<\xi<b. \]
Replacing $f(x)$ by $x^{n+1}$ gives
\[ L(x^{n+1})=\frac{[n+1]_q!}{q^{n(n+1)/2}}\int\limits_{a}^{b}K(x,t)d_{1/q}t, \]
so
\[ \int\limits_{a}^{b}K(x,t)d_{1/q}t=\frac{q^{n(n+1)/2}}{[n+1]_q!}L(x^{n+1}), \]
and this completes the proof.
\end{proof}
\section{Application to polynomial interpolation}
The main idea in this section is to apply the $q$-Peano kernel Theorem on the remainder of polynomial interpolation. Findings demonstrate the advantage of using the $q$-Peano kernel Theorem where the classical theorem does not work.
\begin{proposition}
Suppose $t_0,t_1,\ldots,t_n\in [a,b]$ are distinct  points. For a fixed $x\in [a,b]$, define the corresponding error functional by
\[ L(f)=f(x)-\sum\limits_{k=0}^{n}f(t_k)l_{nk}(x). \]
Then
\[ L(f)=\frac{q^{m(m+1)/2}}{[m]_q!}\sum\limits_{k=0}^{n}l_{nk}(x)\int\limits_{t_k}^{x}(t_k-t)^{m,q}\left(D_{1/q}^{m+1}f\right)(q^mt)d_{1/q}t \quad \textrm{for each } \; m=0,1,\ldots,n. \]
\end{proposition}
\begin{proof}
Since $\sum\limits_{k=0}^{n}l_{nk}(x)=1$, by the $q$-Peano kernel Theorem \ref{peanothm} we get,
\begin{eqnarray*}
\dfrac{[m]_q!}{q^{m(m+1)/2}}K(x,t)=L\left((x-t)_+^{m,q}\right)&=& (x-t)_+^{m,q}-\sum\limits_{k=0}^{n}(t_k-t)_+^{m,q}l_{nk}(x)\\
&=& \sum\limits_{k=0}^{n}\left[(x-t)_+^{m,q}-(t_k-t)_+^{m,q}\right]l_{nk}(x).
\end{eqnarray*} 
From the fact that
\begin{multline}
\int\limits_{a}^{b} \left[(x-t)_+^{m,q}-(t_k-t)_+^{m,q}\right]\left(D_{1/q}^{m+1}f\right)(q^mt)d_{1/q}t=\int\limits_{a}^{x}\left[(x-t)^{m,q}-(t_k-t)^{m,q}\right]\left(D_{1/q}^{m+1}f\right)(q^mt)\\+\int\limits_{t_k}^{x}(t_k-t)^{m,q}\left(D_{1/q}^{m+1}f\right)(q^mt)d_{1/q}t\nonumber
\end{multline}
we have
\begin{multline}
\frac{[m]_q!}{q^{m(m+1)/2}}\int\limits_{a}^{b}K(x,t)\left(D_{1/q}^{m+1}f\right)(q^mt)d_{1/q}t=\int\limits_{a}^{x}\left(D_{1/q}^{m+1}f\right)(q^mt)\sum\limits_{k=0}^{n}\left[(x-t)^{m,q}-(t_k-t)^{m,q}\right]l_{nk}(x)d_{1/q}t\\+\sum\limits_{k=0}^{n}l_{nk}(x)\int\limits_{t_k}^{x}(t_k-t)^{m,q}\left(D_{1/q}^{m+1}f\right)(q^mt)d_{1/q}t.\nonumber
\end{multline}
For each $m\leqslant n$, since the interpolation operator is a projection, it reproduces polynomials and  the term in the square brackets vanishes in the last equation for \linebreak $f(x)=(x-t)^{m,q}$. Accordingly,
\begin{eqnarray*}
L(f)&=& \int\limits_{a}^{b}K(x,t)\left(D_{1/q}^{m+1}f\right)(q^mt)d_{1/q}t\\
&=&\frac{q^{m(m+1)/2}}{[m]_q!}\sum\limits_{k=0}^{n}l_{nk}(x)\int\limits_{t_k}^{x}(t_k-t)^{m,q}\left(D_{1/q}^{m+1}f\right)(q^mt)d_{1/q}t \quad \textrm{for each } \; m=0,1,\ldots,n.
\end{eqnarray*} 
\end{proof}
Now we give examples that show how we can find  the  $q$-Peano kernel.
\begin{example}
 Suppose that we interpolate a function $f\in 1/q-C^{3}[-1,1]$ by a polynomial $p\in \mathcal{P}_2$. Here $n=2$ and $m=2$. Let $t_0=-1,\; t_1=0, \; t_2=1$. Then the error function becomes
\[ L(f)=\frac{q^{3}}{[2]_q!}\sum\limits_{k=0}^{2}l_{2k}(x)\int\limits_{t_k}^{x}(t_k-t)^{2,q}\left(D_{1/q}^{3}f\right)(q^2t)d_{1/q}t \] 
with $l_{20}(x)=\frac{1}{2}x(x-1)$, $l_{21}(x)=(1-x^2)$, $l_{22}(x)=\frac{1}{2}x(x+1).$ Then,
\begin{eqnarray*}
\frac{[2]_q!}{q^{3}}L(f)=&& l_{20}(x)\int\limits_{-1}^{x}(-1-t)^{2,q}\left(D_{1/q}^{3}f\right)(q^2t)d_{1/q}t + l_{21}(x)\int\limits_{0}^{x}(-t)^{2,q}\left(D_{1/q}^{3}f\right)(q^2t)d_{1/q}t \\
&&+l_{22}(x)\int\limits_{1}^{x}(1-t)^{2,q}\left(D_{1/q}^{3}f\right)(q^2t)d_{1/q}t.  
\end{eqnarray*}
Now if $x\leqslant 0$, then
\begin{eqnarray*}
\frac{[2]_q!}{q^{3}}L(f)=& & l_{20}(x)\int\limits_{-1}^{x}(-1-t)^{2,q}\left(D_{1/q}^{3}f\right)(q^2t)d_{1/q}t -l_{21}(x)\int\limits_{x}^{0}(-t)^{2,q}\left(D_{1/q}^{3}f\right)(q^2t)d_{1/q}t \\
& &  -l_{22}(x)\int\limits_{x}^{0}(1-t)^{2,q}\left(D_{1/q}^{3}f\right)(q^2t)d_{1/q}t -l_{22}(x)\int\limits_{0}^{1}(1-t)^{2,q}\left(D_{1/q}^{3}f\right)(q^2t)d_{1/q}t.  
\end{eqnarray*}
Hence,
\[ L(f)=\frac{q^3}{[2]_q!}\int\limits_{-1}^{1}K(x,t)\left(D_{1/q}^{3}f\right)(q^2t)d_{1/q}t \]
where
\[ K(x,t)=\left\{\begin{array}{lr}
l_{20}(x)(-1-t)^{2,q}, & -1\leqslant t \leqslant x\\
&\\
-l_{21}(x)(-t)^{2,q}-l_{22}(x)(1-t)^{2,q}, & x \leqslant t \leqslant 0\\
&\\
-l_{22}(x)(1-t)^{2,q}, & 0\leqslant t \leqslant 1.
\end{array}\right. \]
Similarly for $x\geqslant 0$, 
\begin{eqnarray*}
\frac{[2]_q!}{q^{3}}L(f)=& & l_{20}(x)\int\limits_{-1}^{0}(-1-t)^{2,q}\left(D_{1/q}^{3}f\right)(q^2t)d_{1/q}t +l_{20}(x)\int\limits_{0}^{x}(-1-t)^{2,q}\left(D_{1/q}^{3}f\right)(q^2t)d_{1/q}t \\
& & +l_{21}(x)\int\limits_{0}^{x}(-t)^{2,q}\left(D_{1/q}^{3}f\right)(q^2t)d_{1/q}t  -l_{22}(x)\int\limits_{x}^{1}(1-t)^{2,q}\left(D_{1/q}^{3}f\right)(q^2t)d_{1/q}t  
\end{eqnarray*}
and the Peano kernel becomes
\[ K(x,t)=\left\{\begin{array}{lr}
l_{20}(x)(-1-t)^{2,q}, & -1\leqslant t \leqslant 0\\
&\\
l_{20}(x)(-1-t)^{2,q}+l_{21}(x)(-t)^{2,q}-l_{22}(x)(1-t)^{2,q}, & 0 \leqslant t \leqslant x\\
&\\
-l_{22}(x)(1-t)^{2,q}, & x\leqslant t \leqslant 1.
\end{array}\right. \]
\end{example}
\begin{example}
Let
\[ f(x)=\left\{\begin{array}{lr}
\dfrac{q^3x^3}{6}, & 0 \leqslant x <1\\
&\\
\dfrac{1}{6}\left(4-4[3]_qx+4q[3]_qx^2-3q^3x^3\right), & 1\leqslant x<2\\
&\\
\dfrac{1}{6}\left(-44+20[3]_qx-8q[3]_qx^2+3q^3x^3\right), & 2\leqslant x<3\\
&\\
-\dfrac{1}{6}(-4+x)(-4+qx)(-4+q^2x),& 3\leqslant x<4\\
&\\
0,& \textrm{otherwise.}
\end{array}\right. \] 
It is obvious that for $q\neq 1$, $f\in C[0,4]$ but $f \notin C^1[0,4]$. However, one may check that $f\in 1/q-C^2[0,4]$. Classical error functionals cannot work but we may find the error via the $q$-Peano Kernel theorem. Let $t_0=0$, $t_1=2$ and $t_2=4$. Then the error functional 
\[ L(f)=q\sum\limits_{k=0}^{2}l_{2k}(x)\int\limits_{t_k}^{x}(t_k-t)\left(D_{1/q}^2f\right)(qt)d_{1/q}t \]
where $l_{20}(x)=\frac{1}{8}(x-2)(x-4)$, $l_{21}(x)=-\frac{1}{4}x(x-4)$ and $l_{22}(x)=\frac{1}{8}x(x-2)$. Then,
\begin{eqnarray*}
\frac{1}{q}L(f)=& &  l_{20}(x)\int\limits_{0}^{x}(-t)\left(D_{1/q}^2f\right)(qt)d_{1/q}t +l_{21}(x)\int\limits_{2}^{x}(2-t)\left(D_{1/q}^2f\right)(qt)d_{1/q}t\\
& & + l_{22}(x)\int\limits_{4}^{x} (4-t)\left(D_{1/q}^2f\right)(qt)d_{1/q}t. 
\end{eqnarray*}
Now we will find the kernel. If $0\leqslant x<2$, then
\[ K(x,t)=\left\{\begin{array}{ll}
-l_{20}(x)t,& 0\leqslant t<x\\
&\\
l_{21}(x)(2-t)-l_{22}(x)(4-t),\;\; & x\leqslant t< 2\\
&\\
-l_{22}(x)(4-t), & 2\leqslant t <4.
\end{array}\right. \]
Similarly, for $2\leqslant x<4$,
\[ K(x,t)=\left\{\begin{array}{ll}
-l_{20}(x)t,& 0\leqslant t<2\\
&\\
-l_{20}(x)t+l_{21}(x)(2-t),\;\; & 2\leqslant t< x\\
&\\
l_{21}(x)(2-t)-l_{22}(x)(4-t),\;\; & x\leqslant t <4.
\end{array}\right. \]
\end{example}
The function $f(x)$ given above is indeed a cubic $q$-B-spline.  $q$-B-splines form a basis for quantum splines which are piecewise polynomials whose quantum derivatives agree up to some order at the joins, see \cite{goldman}. \\

\subsection{Trapezoidal rule in $q$-integration}

 Consider the $1/q$-integral of a function $f$ on the interval $[a,b]$. We want to evaluate the $q$-integral approximately using linear interpolant formula. That is,
 
 \[ \int\limits_{a}^{b}f(x)d_{1/q}x\approx\dfrac{b-aq}{[2]_q}f(a)+\dfrac{bq-a}{[2]_q}f(b) \]  
 Let us define the  operator $L$ as
 \[ L(f)=\int\limits_{a}^{b}f(x)d_{1/q}x-\dfrac{b-aq}{[2]_q}f(a)-\dfrac{bq-a}{[2]_q}f(b). \]
 Since $L(f)=0$ for all functions $f\in \mathcal{P}_1$, for  all $f\in 1/q-C^2[a,b]$ we have
 \[ L(f)=\int\limits_{a}^{b}\left(D^2_{1/q}f\right)(qt)K(x,t)d_{1/q}t \]
 and
 \[ K(x,t)=qL((x-t)_+). \]
What follows we find the kernel $K(x,t)$. First,
 \[ K(x,t)=q\left\{\int\limits_{a}^{b}(x-t)_+d_{1/q}x-\dfrac{b-aq}{[2]_q}(a-t)_+-\dfrac{bq-a}{[2]_q}(b-t)_+ \right\}. \]
Then for $t\in [a,b]$,
 \[ \int\limits_{a}^{b}(x-t)_+d_{1/q}x=\int\limits_{t}^{b}(x-t)d_{1/q}x, \quad (a-t)_+=0\quad \textrm{and} \quad  (b-t)_+=(b-t)\]
 Thus,
 \begin{eqnarray*}
K(x,t)&=& q\left\{\int\limits_{t}^{b}(x-t)d_{1/q}x-\dfrac{bq-a}{[2]_q}(b-t)\right\}\nonumber \\
&=& q\left\{\dfrac{(b-t)(b-\frac{t}{q})}{[2]_{1/q}}-\dfrac{bq-a}{[2]_q}(b-t)\right\}\nonumber \\
%&=& q\left\{\dfrac{q(b-t)(b-\frac{t}{q})}{[2]_{q}}-\dfrac{bq-a}{[2]_q}(b-t)\right\}\nonumber\\
&=& \dfrac{q}{[2]_q}(b-t)(a-t)
\end{eqnarray*}
for $a\leqslant t\leqslant b.$\\
Notice that $K(x,t)<0$ on $[a,b]$. Then we can apply Mean Value Theorem \ref{mean}. So, we have
\[ L(f)=\frac{D_{1/q}^2f(\xi)}{[2]_q!}qL(x^2), \]
where
\begin{eqnarray*}
L(x^2)&=& \int\limits_{a}^{b}x^2d_{1/q}x-\dfrac{b-aq}{[2]_q}a^2-\dfrac{bq-a}{[2]_q}b^2\nonumber\\
&=& \dfrac{b^3}{[3]_q!}-\dfrac{a^3}{[3]_q!}-\dfrac{b-aq}{[2]_q}a^2-\dfrac{bq-a}{[2]_q}b^2\nonumber\\
&=& \dfrac{-(b-a)(bq-a)(b-aq)}{[3]_q!}
\end{eqnarray*}
Finally, we derive
\begin{eqnarray*}
L(f)&=& \int\limits_{a}^{b}f(x)d_{1/q}x-\dfrac{b-aq}{[2]_q}f(a)-\dfrac{bq-a}{[2]_q}f(b)\nonumber\\
&=&  \dfrac{-q(b-a)(bq-a)(b-aq)}{[3]_q![2]_q!}D_{1/q}^2f(\xi)
\end{eqnarray*}
where $a<\xi <b.$\\
When $q=1$, the above equation reduces to the well-known trapezoidal rule, see \cite{phillips}.
\subsection{Remainder on quadrature}
We now discuss error bounds of quadrature formulas on remainders given by
\[ R_{n}(f;q)=\int\limits_{0}^{b}f(x)d_{1/q}x-\sum\limits_{k=0}^{n}\gamma_{nk}f(t_{nk}) \]
which appear in numerical integration. Assuming $f\in 1/q-C^{m+1}[0,b]$ and\\ $R_n(f;q)=0$ for all $f \in \mathcal{P}_m$, we can apply the $q$-Peano kernel theorem. Hence
\[ R_{n}(f;q)=\int\limits_{0}^{b}K(x,t)\left(D_{1/q}^{m+1}f\right)(q^mt)d_{1/q}t. \]
By applying the $q$-H\"{o}lder inequality, we have
\[ |R_n(f;q)|\leqslant \left[\int\limits_{0}^{b}\left|\left(D_{1/q}^{m+1}f\right)(q^mt)\right|^{p_1}d_{1/q}t\right]^{\frac{1}{p_1}}\left[\int\limits_{0}^{b}\left|K(x,t)\right|^{p_2}d_{1/q}t\right]^{\frac{1}{p_2}} \]
for all $1\leqslant p_1,p_2\leqslant \infty$ and $\frac{1}{p_1}+\frac{1}{p_2}=1.$ Since the second integral in the above equation is independent of $f$, by choosing coefficients and nodes appropriately we can minimize the remainder.

 \begin{itemize}
\item[(i)] For $p_1=\infty $ and  $p_2=1$,

\[ |R_n(f;q)|\leqslant ||D_{1/q}^{m+1}f||_\infty\int\limits_{0}^{b}|K(x,t)|d_{1/q}t \]
\item[(ii)] For $p_1=p_2=2$,
\[ |R_n(f;q)|\leqslant ||D_{1/q}^{m+1}f||_2\left[\int\limits_{0}^{b}|K(x,t)|^2d_{1/q}t\right]^{\frac{1}{2}}. \]
\end{itemize}
The Peano kernel $K(x,t)$ can be written as
\vspace{0.2cm}

\[ K(x,t)=q^{m(m+3)/2}\dfrac{(b-\frac{t}{q})^{m+1,q}}{[m+1]_q!}-s(t;q), \]

\vspace{0.2cm}

\noindent where $\displaystyle s(t;q)=\frac{q^{m(m+1)/2}}{[m]_q!}\sum_{k=0}^{n}\gamma_{nk}(t_{nk}-t)^{m,q}_+$ is a quantum spline with the knot sequence $\{t_{nk}\}_{k=0,\ldots,n}$. Eventually, the problem of minimizing the $q$-integral
\[ \left[\int_{0}^{b}|K(x,t)|^{p_1}d_{1/q}t\right]^\frac{1}{p_1} \]
is equivalent to finding the best approximation of the polynomial $q^{m(m+3)/2}\dfrac{(b-\frac{t}{q})^{m+1,q}}{[m+1]_q!}$ in $t$ by a quantum spline with respect to the norm $||.||_{p_1}$.

\section{Application to divided differences}
For about a half century, B-splines have played a central role in approximation theory, geometric modeling and wavelets. Recently their $q$-analogues or quantum B-splines has been introduced and studied in \cite{goldman}, \cite{gbudakci}. 

In this section we establish certain relations between $q$-B-splines and $q$-Peano kernels. When $q=1$, Theorem 5.1 reduces to its classical counterpart which can be found in \cite{powel}.

The work \cite{gbudakci} finds that $q$-B-splines of degree $n$ are essentially divided differences of $q$-truncated power functions. That is, the $q$-B-splines are given by

\[ N_{k,n}(t;q)=(t_{k+n+1}-t_k)[t_k,\ldots,t_{k+n+1}](x-t)^{n,q}_+. \]

Now recall the fact that a divided difference $f[t_0,t_1,\ldots,t_{n+1}]$ can be represented as  symmetric sum of $f(t_j)$, see \cite{powel},
\begin{equation}\label{(5.1)}
f[t_0,t_1,\ldots,t_{n+1}]=\sum_{i=0}^{n+1}f(t_{i}
)/
\prod_{j=0\atop j\neq i}^{n+1}(t_{i}-t_{j}).
\end{equation}
\noindent Hence we can readily derive

\[ N_{k,n}(t;q)=(t_{k+n+1}-t_k)\sum_{i=k}^{k+n+1}(t_i-t)_+^{n,q}/\prod\limits_{j=k\atop
j\neq i}^{k+n+1}\frac{1}{(t_i-t_j)}\]
%where
%\[ (t-t_j;q)_+^n= (q^{n-1}t-t_j)\cdots(qt-t_j)(t-t_j)_+.\]

The following theorem is also derived in \cite{gbudakci} by a different method.

\begin{theorem}
\[ f[t_0,t_1,\ldots,t_{n+1}]=\frac{q^{n(n+1)/2}}{[n]_q!}\int_{a}^{b}\frac{N_{0,n}(t;q)}{t_{n+1}-t_0}\left(D_{1/q}^{n+1}f\right)(q^nt)d_{1/q}t. \]
\end{theorem}

\begin{proof}
We first set $L$ as
\begin{align*}
f[t_0,t_1,\ldots,t_{n+1}]=&\sum_{i=0}^{n+1}f(t_{i}
)/
\prod_{j=0\atop j\neq i}^{n+1}(t_{i}-t_{j})\\
=& L(f).
\end{align*}
We see that, for any fixed and distinct points $ \{t_{i}:i=0,1,\ldots,n+1\} $, $ L $ is a bounded linear operator. From the $q$-Peano Kernel Theorem \ref{peanothm}, we have
\[ L(f)=\int_{a}^{b}K(x,t)(D_{1/q}^{n+1}f)(q^nt)d_{1/q}t, \]
where 
\begin{eqnarray*}
K(x,t)&=&\frac{q^{n(n+1)/2}}{[n]_q!}L\left((x-t)_+^{n,q}\right)\\
&=&\frac{q^{n(n+1)/2}}{[n]_q!}\sum_{i=0}^{n+1}(t_{i}-t)_+^{n,q}\left/\right.\prod\limits_{j=0\atop j\neq i}^{n+1}(t_i-t_j).
\end{eqnarray*}
%Since
%\[ (t_i-t)_+^{n,q}=(t_i-t)^{n,q}+(-1)^{n+1}(t-t_i;q)_+^n, \]
%it follows that
Thus
\[ K(x,t)=\frac{q^{n(n+1)/2}}{[n]_q!}\frac{N_{0,n}(t;q)}{t_{n+1}-t_0}. \]
%The term $ L\left((x-t)^{n,q}\right) $ is zero because the function $ \{(x-t)^{n,q};\; a\leqslant x \leqslant b\} $ is in $ \mathcal{P}_{n} $.
Combining the last equation with \eqref{(5.1)} we derive 
\[ f[t_0,t_1,\ldots,t_{n+1}]=\frac{q^{n(n+1)/2}}{[n]_q!}\int_{a}^{b}\frac{N_{0,n}(t;q)}{t_{n+1}-t_0}\left(D_{1/q}^{n+1}f\right)(q^nt)d_{1/q}t. \]
\end{proof}

\noindent\textbf{Acknowledgements}\\
This research is supported by a grant from DEU BAP 2012.KB.FEN.003 and also the first author was supported by a grant (BIDEB-2211) from T\"UBITAK (Scientific and Technological Research Council of Turkey).
%\end{acknowledgements}



\end{document}